\documentclass[reqno,12pt]{amsart}

\usepackage{enumitem}
\usepackage{amssymb}
\usepackage{todonotes}
\usepackage[pdftex,pdfstartview=FitH,pdfborderstyle={/S/B/W 1},%
colorlinks=true, linkcolor=blue, urlcolor=blue, citecolor=blue,%
pagebackref=true]{hyperref}
\usepackage[abbrev,alphabetic]{amsrefs}
\usepackage{url}

\newtheorem{theorem}{Theorem}[section]
\newtheorem*{theorem*}{Theorem}
\newtheorem{lemma}[theorem]{Lemma}
\newtheorem{proposition}[theorem]{Proposition}
\newtheorem{corollary}[theorem]{Corollary}
\newtheorem{OP}[theorem]{Open Problem}

\theoremstyle{definition}
\newtheorem{remark}[theorem]{Remark}

\renewcommand{\leq}{\leqslant}
\renewcommand{\le}{\leq}
\renewcommand{\geq}{\geqslant}
\renewcommand{\ge}{\geq}
\newcommand{\eps}{\varepsilon}

\renewcommand{\hat}{\widehat}

\newcommand\opr[1]{\operatorname{#1}}
\def\R{\mathbf{R}}
\def\C{\mathbf{C}}
\def\Z{\mathbf{Z}}
\def\F{\mathbf{F}}

\def\P{\mathbf{P}}
\def\E{\mathbf{E}}
\def\Var{\opr{Var}}
\def\supp{\opr{supp}}
\newcommand{\wh}{\widehat}

\def\primes{\mathcal{P}} 
\newcommand\muhat[1][]{{\widehat{\mu_{#1}}}} 
\newcommand\pihat[1][]{{\widehat{\pi_{#1}}}} 

\newcommand\floor[1]{\left\lfloor{#1}\right\rfloor}
\newcommand\ceil[1]{\left\lceil{#1}\right\rceil}

\renewcommand{\a}{\alpha}
\renewcommand{\d}{\delta}
\newcommand{\e}{\varepsilon}

\begin{document}

\title[Mixing time of the CDG random process]{Mixing time of the Chung--Diaconis--Graham random process}

\author{Sean Eberhard}
\address{Sean Eberhard, Centre for Mathematical Sciences, Wilberforce Road, Cambridge CB3~0WB, UK}
\email{eberhard@maths.cam.ac.uk}

\author{P\'eter P.~Varj\'u}
\address{P\'eter P.~Varj\'u, Centre for Mathematical Sciences, Wilberforce Road, Cambridge CB3~0WB, UK}
\email{pv270@dpmms.cam.ac.uk}

\thanks{Both authors have received funding from the European Research Council (ERC) under the European Union’s Horizon 2020 research and innovation programme (grant agreement No. 803711). PV was supported by the Royal Society.}

\begin{abstract}
Define $(X_n)$ on $\Z/q\Z$ by $X_{n+1} = 2X_n + b_n$, where the steps $b_n$ are chosen independently at random from $-1, 0, +1$. The mixing time of this random walk is known to be at most $1.02 \log_2 q$ for almost all odd $q$ (Chung--Diaconis--Graham, 1987), and at least $1.004 \log_2 q$ (Hildebrand, 2009). We identify a constant $c = 1.01136\dots$ such that the mixing time is $(c+o(1))\log_2 q$ for almost all odd $q$.

In general, the mixing time of the Markov chain $X_{n+1} = a X_n + b_n$ modulo $q$, where $a$ is a fixed positive integer and the steps $b_n$ are i.i.d.~with some given distribution in $\Z$, is related to the entropy of a corresponding self-similar Cantor-like measure (such as a Bernoulli convolution). We estimate the mixing time up to a $1+o(1)$ factor whenever the entropy exceeds $(\log a)/2$.
\end{abstract}

\maketitle

\section{Introduction}

Let $a > 1$ be a positive integer and let $\mu$ be a finitely supported measure on $\Z$. Assume that $\gcd(\supp \mu - \supp \mu) = 1$, i.e., that $\mu$ is not supported on a coset of a proper subgroup. Define a random process $(X_n)_{n\geq 0}$ on $\Z$ by $X_0 = 0$ and
\[
  X_{n+1} = a X_n + b_n,
\]
where $b_1, b_2, \dots$ are independent and distributed according to $\mu$. Let $\mu_n$ be the law of $X_n$. In other words, $\mu_n$ is defined by the recursion
\begin{align*}
  \mu_0 &= \delta_0, \\
  \mu_n &= \mu * (m_a)_* \mu_{n-1} \qquad (n\geq 1),
\end{align*}
where $m_a(x) = ax$ is the multiplication-by-$a$ map.

Now for any positive integer $q$ consider the reduction $X_n \bmod q$. The mixing time of this random process, loosely speaking the smallest $n$ such that
\[
  \|\mu_n \bmod q - u_q\| = o(1),
\]
where $u_q$ is the uniform measure on $\Z / q\Z$ and $\|\cdot\|$ denotes $l^1$ norm, was first studied in detail by Chung, Diaconis, and Graham~\cite{CDG} (motivated by pseudorandom number generation). In the prototypical case $a = 2$, $\mu = u_{\{-1, 0, 1\}}$ (uniform on $\{-1, 0, 1\}$), the following bounds have been proved.
\begin{itemize}
  \item If $n \geq (c_1 + \eps) \log_2 q$, where
  \[
    c_1 = \left( 1 - \log_2 \left( \frac{5 + \sqrt{17}}{9} \right)\right)^{-1} \approx 1.02, 
  \]
  then
  \[
    \frac12 \|\mu_n \bmod q - u_q\| \leq \eps
  \]
  for almost all odd $q$ (Chung--Diaconis--Graham~\cite{CDG}).
  \item If $n \leq c_2 \log_2 q$, where
  \[
    c_2 = \frac1{-\frac12 - \frac4{18} \log_2(\frac{4}{36}) - \frac{5}{18} \log_2(\frac{5}{36})} \approx 1.004, 
  \]
  then
  \[
    \frac12 \|\mu_n \bmod q - u_q\| \to 1
  \]
  as $q \to\infty$ (Hildebrand~\cite{hildebrand}).
  For further work in this direction, see \cite{Hil2}, \cite{Nev} and the references therein.
\end{itemize}
Thus the mixing time of $X_n \bmod q$ is a slight constant multiple larger than $\log_2 q$, for typical $q$. (Chung, Diaconis, and Graham also show that there are infinitely many odd $q$ for which the mixing time exceeds $c \log q \log \log q$.)

We advance the following explanation of this phenomenon. Defining the entropy of any measure $\mu$ by
\[
  H(\mu) = \sum_x \mu(\{x\}) \log \mu(\{x\})^{-1},
\]
the \emph{asymptotic entropy} (or \emph{entropy rate}) of $(X_n)_{n\geq0}$ is defined by
\[
  H(a, \mu) = \lim_{n\to\infty} \frac1n H(\mu_n).
\]
It is well known that the limit exists, and that
\[
  0 < H(a, \mu) \leq \min(\log a, H(\mu)).
\]
In the case of $\mu = u_{\{-1, 0, 1\}}$, a computation shows that the entropy $H(2, \mu)$ is slightly less than $\log 2$. On the other hand, $X_n \bmod q$ cannot equidistribute before the entropy of $\mu_n$ reaches $(1-o(1))\log q$. Hence it takes slightly more than $(\log q) / (\log 2)$ steps for $X_n \bmod q$ to mix.

In general, the entropy $H(a, \mu)$ and the mixing time of $X_n \bmod q$ are related by our main theorem, provided that $H > \frac12 \log a$.

\begin{theorem} \label{mainthm}
Let $a$ and $\mu$ be given. Let $H = H(a, \mu)$. The following hold for all $\d>0$.
\begin{enumerate}[label=(\roman*)]
  \item If $Hn \leq (1-\delta) \log q$ then
  \[
    \frac{1}{2}\|\mu_n \bmod q - u_q\| = 1 - o(1).
  \]
  \item Assume $H > \frac12 \log a$. Then for a proportion $1-o(1)$ of primes $q$ such that $\log q \leq (1-\delta) Hn$, we have
  \[
    \|\mu_n \bmod q - u_q \| = o(1).
  \]
  The same holds for a proportion $1-o(1)$ of the composite $q$ such that $(q, a) = 1$.
\end{enumerate}
\end{theorem}

We have already sketched the proof of the first part. The proof of the second part consists of the following ideas:
\begin{enumerate}
  \item By a version of the Shannon--McMillan--Breiman theorem (a.k.a.\ asymptotic equipartition), there is a set $S$ of $\mu_n$-measure $1-o(1)$ such that
  \[
    \mu_n(\{x\}) = \exp(-(H + o(1)) n)
  \]
  for all $x \in S$. In particular, there is a measure $\nu_n$ (the conditional measure on $S$) such that $\|\mu_n - \nu_n\| = o(1)$ and \[\log \|\nu_n\|_2^{-2} = (H + o(1))n.\]
  \item By the large sieve, for most primes $q \asymp |S|$, the $l^2$ norm of $\nu_n \bmod q$ is not much larger than the $l^2$ norm of $\nu_n$. Here we use the hypothesis that $H > \frac{1}{2} \log a$.
  \item Therefore for most primes $q \asymp |S| = \exp((H + o(1))n)$, $\mu_n \bmod q$ is close to uniform, so the mixing time modulo $q$ is at most $n$.
\end{enumerate}
The proof in the composite case also uses the large sieve, but in a more involved way.

We believe that the hypothesis $H > \frac12 \log a$ is unnecessary.  It might be possible to relax this to $H \ge \frac12 \log a$ by a suitable modification of our proof. We briefly sketch some ideas towards this in Section \ref{sc:OP}, but we do not pursue this in detail. The problem of removing the hypothesis entirely is related to some open questions about the large sieve: it is believed that the only examples for which the large sieve is sharp are somehow quadratic (Green--Harper~\cite{green--harper}), but the support of $\mu_n$ cannot be anything like quadratic. Although we have not succeeded in removing this hypothesis, we can adapt Gallagher's larger sieve to prove the following bound which is weaker by a factor of $2$, and holds only for a looser sense of ``almost all''.

\begin{theorem} \label{larger-sieve-thm}
Let $\eps > 0$, and let $\primes_n$ be the set of primes $p$ such that
\[
  \log p \leq (1/2-\eps) Hn
\]
and
\[
  \|\mu_n \bmod p - u_p \| \geq \eps.
\]
Then
\[
  \sum_{{p \in \primes_n}} \frac{\log p}{p} \ll_\eps n^{1/2}.
\]
\end{theorem}

By contrast, Mertens' first theorem states
\[
  \sum_{\log p \leq X} \frac{\log p}{p} = X + O(1).
\]
Thus the theorem asserts $\primes_n$ has asymptotically zero density among primes $p$ such that $\log p \leq (1/2-\eps)Hn$ when the primes are weighted by $(\log p)/p$. Equivalently, the mixing time of $X_n \bmod p$ is at most $(2+o(1)) H^{-1} \log p$ for almost all primes in almost all dyadic intervals.

The entropy $H(a, \mu)$ is equivalent, by reversal and rescaling, to the entropy of the self-similar Cantor-like measure $\lambda$ on $\R$ defined by
\[
  \lambda = \mu * (m_{1/a})_* \lambda,
\]
where $m_{1/a}(x) = x/a$. This general class of measures includes the case of Bernoulli convolutions, which were studied extensively by Erd\H{o}s, Garsia, and others. In the particular case $a=2, \mu = u_{\{-1, 0, 1\}}$, Hare, Hare, Morris, and Shen~\cite{hhms}*{Theorem~3.1} prove that
\[
  \frac{H}{\log 2} = \log 3 - \frac{3}{2} L(1/3),
\]
where $L$ is the analytic function in the unit disk defined by
\[
  L(z) = (1 - 3z)^2 \sum_{n=1}^\infty z^n \sum_{\substack{1 \leq i < j \\ \gcd(i, j) = 1 \\ e(i, j) = n}} j \log j,
\]
where $e(i, j)$ is the number of subtractions made by the Euclidean algorithm applied to $(i, j)$ (the singularity $z = 1/3$ is removable). They deduce
\[
  \frac{H}{\log 2} = 0.9887658714\dots.
\]
Since in particular this is greater than $0.5$, Theorem~\ref{mainthm} applies, so the mixing time of $X_n \bmod q$ is almost always (and never less than)
\[
  (1.01136176816\dots) \log_2 q.
\]

\subsection{Notation}

We write $\|\mu\|$ for the standard norm for a complex measure $\mu$, so that the maximum distance between two probability measures $\mu$ and $\nu$ is $2$ (many authors divide this norm by $2$).
When we write $\|\mu\|_2$ or $H(\mu)$ we are implicitly identifying $\mu$ as the function $\mu(\{x\})$, its density with respect to the discrete measure, so that
\begin{align*}
  \|\mu\|_2^2 &= \sum_x \mu(\{x\})^2, \\
  H(\mu) &= \sum_x \mu(\{x\}) \log \mu(\{x\})^{-1}.
\end{align*}

Likewise we normalize the Fourier transform with the discrete measure on the physical side $\Z$ and the uniform measure on the frequency side $\R/\Z$; thus
\begin{align*}
  \muhat(\xi) &= \sum_{x \in \Z} \mu(\{x\}) e(-\xi x) = \int e(-\xi x) \, d\mu(x), \\
  \mu(\{x\}) &= \int_\xi \muhat{}(\xi) e(\xi x) \, d\xi,
\end{align*}
where as usual $e(x) = e^{i 2 \pi x}$.
We keep this normalization when we reduce modulo $q$, so the Fourier transform of $\mu \bmod q$ is just $\muhat(r/q)$ $(0\leq r < q)$. The inversion formula modulo $q$ is
\[
  (\mu \bmod q)(\{x\}) = \mu(x + q\Z) = \frac1q \sum_{r=0}^{q-1} \muhat(r/q) e(rx/q).
\]
We note that the above choice of definition of the Fourier transform differs from that in \cite{CDG} in a sign.

We use standard asymptotic notation from analytic number theory where expedient: $X \ll Y$ or equivalently $X = O(Y)$ means $X \leq CY$ for some constant $C>0$; $X \asymp Y$ means $X \ll Y$ and $Y \ll X$; $X\sim Y$ means $X/Y\to 1$ as $n\to \infty$. Constants $C, c$, etc, are usually allowed to depend on $\mu$ and $a$.

\subsection{Organization of the paper}

We discuss the Shannon--McMillan--Breiman theorem for random walks (and the rate of convergence) in Section~\ref{sc:SMB}.
The lower bound in Theorem~\ref{mainthm} is proved in Section~\ref{sc:lower}.
In Section~\ref{sc:small-moduli}, we revisit an argument in \cite{CDG} to prove a crude upper bound on the mixing time which is valid for all moduli.
This will be used in later sections in the proof of the upper bound on the mixing time in Theorem~\ref{mainthm} to treat small divisors of composite moduli.
In Sections~\ref{sc:large-sieve} and \ref{sc:alternative}, we give two different proofs of the upper bound in Theorem~\ref{mainthm}.
We give the proof of Theorem~\ref{larger-sieve-thm} in Section~\ref{sc:larger-sieve}.
We conclude the paper with some open problems in Section~\ref{sc:OP}.

\section{SMB via Efron--Stein}\label{sc:SMB}

We need a version of the Shannon--McMillan--Breiman theorem which states that
\[
  \mu_n(\{x\}) = \exp(-H(\mu_n) + o(n))
\]
for asymptotically $\mu_n$-almost-all $x$.

\begin{theorem} \label{SMB}
There is a constant $C$ depending only on $\mu$ such that
\[
  \mu_n\left( \{ x: |-\log \mu_n(\{x\}) - H(\mu_n)| \geq \a n \}\right) \leq \frac{C}{\a^2 n}
\]
for all $n\geq 1$ and $\a > 0$.
\end{theorem}

The Shannon--McMillan--Breiman theorem for a random walk on a group is due to
Derriennic~\cite{Der-sub-additive}*{Section IV}, and Kaimanovich and Vershik~\cite{KV-RW}*{Theorem 2.1}.
For our application, we need an estimate for the rate of convergence, and for this reason we give a proof.
Our proof exploits special features of the walk and is based on the Efron--Stein inequality.

Fix some $n\geq 1$, and let $b_1, \dots, b_n$ be independent $\mu$-distributed random variables. For $m \in \Z$ and $j \in \{1, \dots, n\}$, let
\begin{align*}
  g(m) &= \P\left(\sum_{i=1}^n b_i a^i = m\right),\\
  g_j(m) &= \P\left(\sum_{i\neq j} b_ia^i = m\right).
\end{align*}
(Note that $g(x) = \mu_n(\{x\})$.) For $x_1, \dots, x_n \in \Z$ define
\begin{align*}
  f(x_1,\ldots,x_n) &=\log g\left(\sum_{i=1}^n x_ia^i\right)\\
  f_j(x_1,\ldots,x_n) &=\log g_j\left(\sum_{i\neq j} x_ia^i\right).
\end{align*}
Our aim is show that the random variable
\[
  Z = f(b_1, \dots, b_n).
\]
is concentrated around its mean
\[
  \E[Z] = -H(\mu_n).
\]
We will apply the Efron--Stein inequality (in the asymmetric case due to Steele~\cite{steele--efron--stein}). For $j\in\{1,\ldots,n\}$, let
\[
  Z_j = f(b_1, \dots, b_{j-1}, b_j', b_{j+1}, \dots, b_n),
\]
where $b_j'$ is an independent copy of $b_j$. The Efron--Stein inequality states
\begin{equation} \label{efron--stein}
  \Var[Z]\le\frac12\sum_{j=1}^{n}\E[(Z-Z_j)^2].
\end{equation}

\begin{lemma}\label{lem:variance}
For $j \in \{1,\dots,n\}$ we have
\[
  \E[(Z-f_j(b_1,\dots,b_n))^2]
  = \E[(Z_j-f_j(b_1,\dots,b_n))^2]
  \le C,
\]
where $C$ is a constant depending only on $\mu$. In particular,
\[
  \E[(Z - Z_j)^2] \leq 4C.
\]
\end{lemma}
\begin{proof}
The equality in the claim holds because $f_j$ does not depend on the $j$th variable and $b_j'$ is an independent copy of $b_j$ (so in fact $Z - f_j(b_1, \dots, b_n)$ and $Z_j - f_j(b_1, \dots, b_n)$ are identically distributed). The last statement will follow by
\begin{align*}
  \E[(Z - Z_j)^2]
  &\leq 2 \left( \E[(Z - f_j(b_1, \dots, b_n))^2] + \E[(Z_j - f_j(b_1, \dots, b_n))^2] \right) \\
  &\leq 4C.
\end{align*}
Hence it suffices to prove the main claim that
\[
  \E[(Z - f_j(b_1, \dots, b_n))^2] \leq C.
\]

Let $m\in\Z$ and $x\in\supp b_j$. First,
\[
  g(m+xa^j)\ge\P(b_j=x)\P\left(\sum_{i\neq j}b_ja^j=m\right)\ge c g_j(m),
\]
where $c>0$ is a constant depending only on $\mu$, namely,
\[
  c = \min_{x \in \supp \mu} \mu(\{x\}).
\]
Second,
\begin{align*}
  g(m+xa^j)
  &= \sum_{y \in \supp b_j} \P(b_j=y)\P\left(\sum_{i\neq j}b_ia^i=m+(x-y)a^j\right)\\
  &\le \max_{y\in\supp b_j} g_j(m+(x-y)a^j).
\end{align*}
Thus, for $(x_1, \dots, x_n) \in (\supp \mu)^n$,
\[
  f(x_1,\ldots,x_n) \ge f_j(x_1,\ldots,x_n)+\log c,
\]
and if
\[
  f(x_1,\ldots,x_n) \ge f_j(x_1,\ldots,x_n)+t
\]
for some $t\ge 0$ then
\[
  \sum_{i\neq j} x_ia^i \in M_t,
\]
where
\[
  M_t = \{m\in\Z:\exists x\in\supp (b_1-b_2): g_j(m)\le e^{-t} g_j(m+xa^j)\}.
\]

We show that
\begin{equation}\label{eq:ESclaim}
  \P\left(\sum_{i\neq j} b_ia^i \in M_t\right)
  = \sum_{m\in M_t} g_j(m)\le|\supp (b_1-b_2)| e^{-t}.
\end{equation}
We consider the set
\[
  M_{t,x} = \{m\in\Z: g_j(m)\le e^{-t} g_j(m+xa^j)\}
\]
separately for each $x\in\supp (b_1-b_2)$. Since $\sum g_j = 1$ we have
\[
  \sum_{m \in M_{t, x}} g_j(m) \leq e^{-t} \sum_{m \in M_{t,x}} g_j(m + xa^j) \leq e^{-t}.
\]
Summing over $x \in \supp(b_1 - b_2)$, we get the claim \eqref{eq:ESclaim}.

Combining our estimates we have
\begin{equation} \label{Z-tail-bound}
  \P(|Z-f_j(b_1,\dots,b_n)|>t)< C e^{-t}
\end{equation}
for each $t>C$, where $C$ is a constant depending only on $\mu$. Thus indeed
\begin{align*}
  \E[(Z - f_j(b_1, \dots, b_n))^2]
  &= 2 \int_0^\infty t\P(|Z - f_j(b_1, \dots, b_n)| > t) \, dt \\
  &\leq C'.
\end{align*}
This completes the proof of the lemma.
\end{proof}

We can now prove Theorem~\ref{SMB}. By Lemma~\ref{lem:variance} and the Efron--Stein inequality \eqref{efron--stein}, we have $\Var[Z] \leq 2Cn$. Hence by Chebyshev's inequality
\[
  \P(|Z - \E[Z]| \geq \a n) \leq \frac{2C}{\a^2 n}.
\]
By definition, $Z$ has the same distribution as $\log \mu_n(\{X_n\})$ and $\E[Z] = -H(\mu_n)$. This proves the theorem.

\begin{corollary} \label{cor:nu_n-exists}
Let $\a = \a(n) > 0$, with $\a^2 n \to \infty$. There is a probability measure $\nu_n$ such that
\begin{enumerate}
  \item $\|\mu_n - \nu_n\| \ll 1 / (\a^2 n)$,
  \item $\supp \nu_n \subset \supp \mu_n$,
  \item $\nu_n(\{x\}) = \exp(-H(\mu_n) + O(\alpha n))$ for every $x \in \supp \nu_n$.
\end{enumerate}
\end{corollary}

The corollary follows by defining
\begin{align*}
  &S = \{x : |-\log\mu_n(\{x\}) - H(\mu_n)| \leq \a n\}, \\
  &\nu_n(A) = \mu_n(A \cap S) / \mu_n(S).
\end{align*}
By Theorem~\ref{SMB},
\[
  \|\mu_n - \nu_n\| = 2\mu_n(S^c) \ll \frac1{\alpha^2n}.
\]

\begin{remark}
The key estimate \eqref{Z-tail-bound} is a decent tail bound, so it may be possible to get something much stronger with more sophisticated concentration inequalities. See for example the book of Boucheron, Lugosi, and Massart~\cite{blmbook}*{Section~6.9}. Persuing this could lead to an improved estimate for the size of the cut-off window in the mixing of $X_n \bmod q$.
\end{remark}

\section{The lower bound}\label{sc:lower}

In this section we prove Theorem~\ref{mainthm}(i), which states that the mixing time of $X_n \bmod q$ is at least $H^{-1} \log q$ for all sufficiently large $q$. More precisely, if $H n \leq (1-\delta) \log q$ for some constant $\delta > 0$ then
\[
  \frac12 \|\mu_n \bmod q - u_q\| \to 1 \qquad (q \to \infty).
\]
By Corollary~\ref{cor:nu_n-exists}, there is a measure $\nu_n$ such that
\begin{align*}
  &\|\mu_n - \nu_n\| \ll 1 / (\a^2 n), \\
  &|\supp \nu_n| \leq \exp(H(\mu_n) + O(\a n)).
\end{align*}
From the bound on $|\supp \nu_n|$,
\begin{align*}
  \frac12 \|\nu_n \bmod q - u_q\|
  &\geq u_q((\supp \nu_n \bmod q)^c) \\
  &\geq 1 - \exp(H(\mu_n) + O(\alpha n) - \log q).
\end{align*}
Now $H(\mu_n) = Hn + o(n)$, and $\log q - H n \geq \delta \log q \geq \delta H n$. Hence, if $\alpha$ is a constant sufficiently smaller than $\delta H$,
\[
  \frac12 \|\nu_n \bmod q - u_q\| = 1 - e^{-cn},
\]
so
\[
  \frac12 \|\mu_n \bmod q - u_q\| = 1 - O(1/n).
\]

\begin{remark}
To prove just that $\|\mu_n \bmod q - u_q\|$ does not tend to zero, we do not need the SMB theorem. Suppose $H(\mu_n) \leq (1-\delta) \log q$. By Markov's inequality,
\[
  \mu_n(\{x : \log \mu_n(\{x\})^{-1} > \log q - 10 \}) \leq \frac{H(\mu_n)}{\log q - 10} \leq 1-\delta + o(1).
\]
Therefore $\mu_n(\{x\}) \geq e^{10} / q$ on a set $S$ of $\mu_n$-measure at least $\delta + o(1)$, so
\[
  \frac12 \| \mu_n \bmod q - u_q \| \geq (1 - e^{-10}) (\delta + o(1)).
\]
\end{remark}

\section{Small moduli}\label{sc:small-moduli}

In the next section we will need to know that $\mu_n \bmod q \approx u_q$ for any modulus $q$ such that $(q, a) = 1$ and $q = o(\log n / \log \log n)$. In other words, the mixing time of $X_n \bmod q$ is at most $O(\log q \log \log q)$ for all $q$. This was proved in \cite{CDG}*{Section~4} in the model case $a = 2$, $\mu = u_{\{-1, 0, 1\}}$, which was later generalized by Hildebrand in \cite{Hil-thesis}*{Chapter 3}
and \cite{Hil0}*{Theorem 2}. We state the result in the form we will use it.

\begin{lemma} \label{lem:small-moduli}
Assume $(q, a) = 1$, and that $n > C \log q \log \log q$ for a sufficiently large constant $C$ (depending on $\mu$ and $a$). Then
\[
  q\| \mu_n \bmod q - u_q \|_2^2 \leq e^{-c n / \log q}
\]
for some constant $c>0$.
\end{lemma}

While all the ideas required for the proof are already in \cites{CDG, Hil-thesis, Hil0}, we include the proof for the reader's convenience.

\begin{proof}
Since by assumption $\mu$ is not supported on a coset of a proper subgroup, $\muhat(\xi)=1$ only if $\xi = 0$ (for $\xi \in \R/\Z$). Moreover, for a $\mu$-distributed random variable $X$,
\begin{align*}
  \muhat'(0) &= \partial_{\xi=0} \E e(\xi X) = i 2 \pi \E X, \\
  \muhat''(0) &= \partial^2_{\xi=0} \E e(\xi X) = -(2 \pi)^2 \E X^2 < 0,
\end{align*}
so there is some constant $c>0$ such that
\[
  |\muhat(\xi)| \leq e^{-c |\xi|^2},
\]
where $|\xi|$ denote the distance from $\xi$ to $0$ in $\R/\Z$.

In general, for $n \geq 1$ we have
\begin{equation} \label{muhatn-formula}
  \muhat[n](\xi) = \prod_{i=0}^{n-1} \muhat(a^i \xi).
\end{equation}
Let
\[
  \xi = 0.\xi_1\xi_2\xi_3\dots \qquad (0 \leq \xi_i \leq a-1)
\]
be the base-$a$ expansion of $\xi$. Then
\[
  |\muhat[n](\xi)| \leq e^{-c' A_n(\xi)},
\]
where $A_n(\xi)$ is the number of indices $i \in \{1, \dots, n-1\}$ such that we do \emph{not} have
\[
  \xi_i = \xi_{i+1} \in \{0, a-1\},
\]
as in all other cases we have $|a^{i-1} \xi| \geq 1/a^2$.

Let $n_0 = \ceil{\log_a q}$ and $n = kn_0$ for some $k\in\Z_{>0}$. By H\"older's inequality,
\begin{align*}
  q\|\mu_n \bmod q - u_q \|_2^2
  &= \sum_{r\neq 0} |\muhat[n](r/q)|^2 \\
  &= \sum_{r\neq 0} |\muhat[n_0](r/q)|^2 \cdots |\muhat[n_0](a^{(k-1)n_0} r/q)|^2 \\
  &\leq \sum_{r\neq 0} |\muhat[n_0](r/q)|^{2k}.
\end{align*}
Note $|r/q - r'/q| \geq 1/q \geq a^{-n_0}$ for $r \neq r'$, and in particular $|r/q| \geq a^{-n_0}$ for $r \neq 0$. Thus every $r/q$ is identified by its first $n_0$ digits $\xi_1, \dots, \xi_{n_0}$, and unless $r=0$ we cannot have $\xi_1 = \cdots = \xi_{n_0} \in \{0, a-1\}$. The number of $\xi = 0.\xi_1\cdots\xi_{n_0}$ with $A_{n_0}(\xi) = A$ is bounded by
\[
  \binom{n_0}{A} a^{A+1},
\]
as unless $i$ counted by $A_{n_0}(\xi)$ then $\xi_{i+1} = \xi_i$. Hence
\begin{align*}
  q\|\mu_n \bmod q - u_q \|_2^2
  &\leq \sum_{\xi = 0.\xi_1 \cdots \xi_{n_0} \neq 0} e^{-c' A_{n_0}(\xi) k} \\
  &\leq \sum_{A=1}^{n_0} \binom{n_0}{A} a^{A+1} e^{-c'A k} \\
  &= a((1 + a e^{-c' k})^{n_0} - 1) \\
  &\leq a(\exp(a e^{-c'k} n_0) - 1).
\end{align*}
Provided that $k > C \log \log q$ for a sufficiently large constant $C$, this is bounded by
\[
  O(a^2 e^{-c' k} n_0) \le e^{-c'' n / \log q},
\]
as claimed.
\end{proof}

\section{The large sieve argument}\label{sc:large-sieve}
The purpose of this section is to prove Theorem~\ref{mainthm}(ii).

\def\M{\mathcal{M}}
\newcommand{\piop}[1]{\pi_{#1}}
\newcommand{\Pop}[2][q_0]{P_{#2}^{(#1)}}

For any group $G$, let $\M(G)$ be the space of complex measures on $G$. For $q_1 \mid q$, we consider $\M(\Z/q_1\Z)$ to be the subspace of $\M(\Z/q\Z)$ or $\M(\Z)$ consisting of measures which are uniform on fibres mod $q_1$. For integers $q_0 \mid q$, we define two projection operators
\[
  \piop{q}, \Pop{q} : \M(\Z) \to \M(\Z/q\Z).
\]
The easiest way to give their definition is in terms of Fourier multipliers: for $(r,s) = 1$, $s \mid q$,
\begin{align*}
  \widehat{\piop{q} \nu}(r/s)
  &= \widehat{\nu}(r/s) 1_{s \mid q} \\
  \widehat{\Pop{q} \nu}(r/s)
  &= \widehat{\nu}(r/s) 1_{[s, q_0] = q}.
\end{align*}
Notice that
\begin{align*}
  \piop{q} \nu &= \nu \bmod q \\
  &= \sum_{s: q_0 \mid s \mid q} \Pop{s} \nu,\\
  \Pop{q} \nu &= \sum_{s : q_0 \mid s \mid q} \mu(s / q_0) \piop{s} \nu,
\end{align*}
where $\mu$ is the M\"obius function.
From these identities, we deduce that $\|\piop{q}\| = 1$ and $\|\Pop{q}\| \leq d(q/q_0)$,
where $\|\cdot\|$ stands for the operator norm with respect to the $l^1$ norm, and $d(\cdot)$
stands for the divisor function.

\begin{lemma} \label{lem:large-sieve}
Assume $2H > (1+\eps) \log a$. Let $Q$ be large, and let $n$ be large enough that
\[
  H(\mu_n) > (1 + \eps) \log Q.
\]
Then for all but at most $Q^{1-\eps/3}$ many $q \in [Q/2, Q]$, and for all $q_0 < q^{\eps / 3}, q_0 \mid q$, we have
\[
  \|\Pop{q} \mu_n\| \ll \frac{\|\Pop{q}\|}{\log q}.
\]
\end{lemma}

Note that if $q$ is prime then $\Pop[1]{q} = \pi_q - \pi_1$, and
\[
  \Pop[1]{q} \mu_n = \mu_n \bmod q - u_q.
\]
Thus the prime case of Theorem~\ref{mainthm}(ii) follows immediately. The composite case will require more work, which we explain at the end of this section.

The proof of the lemma relies on the large sieve inequality, which we recall now.
Letting $q$ be an integer, we write $R(q)$ for the residues modulo $q$ that are coprime to $q$.

\begin{theorem}[\cite{friedlander--iwaniec--opera}*{Theorem 9.3}]\label{th:large-sieve}
Let $f:\Z\to\C$ be a function supported on $[-N,N]$ for some $N\in\Z_{>0}$.
Then
\[
\sum_{q\le Q}\sum_{r\in R(q)} |\wh f(r/q)|^2\le (Q^2+2N)\sum_n |f(n)|^2.
\]
\end{theorem}

We record the following corollary.
\begin{corollary}\label{cr:large-sieve}
Let $q_0,Q,N\in\Z_{>0}$, and let $\nu$ be a probability measure supported on $\Z\cap[-N,N]$.
Then
\[
\sum_{\substack{q \in [Q/2, Q]\\ q_0|q}} q\|P_q^{(q_0)}\nu\|_2^2\ll(Q^2+N)\|\nu\|_2^2.
\]
\end{corollary}

\begin{proof}
By definition, we have
\[
q\|P_q^{(q_0)}\nu\|_2^2=\sum |\wh\nu(r/s)|^2,
\]
where the summation on the right extends over all fractions $r/s$
such that $s \mid q$, $[s,q_0]=q$ and $r\in R(s)$.
We sum this up for $q \in [Q/2, Q]$ and $q_0 \mid q$, and note that each fraction $r/s$ will appear
at most once on the right hand side.
Therefore, we can invoke Theorem~\ref{th:large-sieve}, which yields the desired inequality.
\end{proof}

\begin{proof}[Proof of Lemma \ref{lem:large-sieve}]
Note that $\|\Pop{q}\mu_n\|$ is monotonic in $n$, as $\Pop{q}$ commutes with convolution:
\begin{align*}
  \|\Pop{q}\mu_{n+1}\|
  &= \|\Pop{q} (\mu_n * (m_{a^n})_* \mu)\|\\
  &= \|(\Pop{q}\mu_n) * (m_{a^n})_* \mu\|\\
  &\leq \|\Pop{q}\mu_n\|.
\end{align*}
Hence, by reducing $n$ if necessary, and by assuming $Q$ is large enough, we may assume that
\[
  H(\mu_n) \sim H n \sim (1 + \eps) \log Q.
\]
In particular, since $2H > (1+\eps) \log a$, we have $Q^2 > a^n$.

By Corollary~\ref{cor:nu_n-exists}, there is a measure $\nu_n$ such that $\|\mu_n - \nu_n\| \ll \a^{-2}/n$ and
\[
  \log \|\nu_n\|_2^{-2} = H(\mu_n) + O(\a n) \sim (1 + \eps) \log Q.
\]
Fix any $q_0 < Q^{\eps / 3}$. The measures $\Pop{q} \nu_n$ for various $q < Q, q_0 \mid q$, have disjoint Fourier supports. Therefore, by the large sieve, Corollary \ref{cr:large-sieve},
\[
  \sum_{\substack{q \in [Q/2, Q] \\ q_0 \mid q}} q \|\Pop{q} \nu_n\|_2^2 \ll (a^n + Q^2) \|\nu_n\|_2^2 \ll Q^{1-\eps}.
\]
Hence with at most $Q^{1-2\eps/3}$ exceptions we have
\[
  q \|\Pop{q} \nu_n\|_2^2 \ll Q^{-\eps/3}.
\]
Since there are at most $Q^{1-2\eps/3}$ exceptions for each $q_0 < Q^{\eps / 3}$, there are at most $Q^{1-\eps/3}$ exceptions in all.

For unexceptional $q$ we have, for any $q_0 < Q^{\eps / 3}, q_0 \mid q$,
\begin{align*}
  \|\Pop{q} \mu_n\|
  &\leq \|\Pop{q}(\mu_n - \nu_n)\| + \|\Pop{q} \nu_n\|\\
  &\leq \|\Pop{q}(\mu_n - \nu_n)\| + q^{1/2} \|\Pop{q} \nu_n\|_2\\
  &\ll \|\Pop{q}\| \alpha^{-2} n^{-1} + Q^{-\eps/6}.
\end{align*}
If $\alpha$ is a sufficiently small constant we deduce
\[
  \|\Pop{q} \mu_n\| \ll \frac{\|\Pop{q}\|}{n} \asymp \frac{\|\Pop{q}\|}{\log Q},
\]
as claimed.
\end{proof}

For the proof of the composite case of Theorem~\ref{mainthm}(ii), we need one more ingredient, a result about the prime factorization of a typical integer.

\begin{lemma}\label{lm:divisors}
All but $o(Q)$ integers $q\in[1,Q]$ can be written as
\[
q=p_1p_2\cdots p_k q_0,
\]
where $p_1>\cdots>p_k$ are primes and
\begin{align*}
  k &\asymp \log \log \log Q,\\
  \log q_0 &\in \left[\frac{\log Q}{(\log \log Q)^3}, \frac{\log Q}{(\log \log Q)^2}\right].
\end{align*}
\end{lemma}

(The exponents $2$ and $3$ are not special: we could replace the interval for $\log q_0$ with any logarithmically long interval, adjusting $k$ as necessary.)

This follows easily from well-known results on the ``anatomy of integers''.
For the sake of completeness we give the proof.

\begin{proof}
We first apply \cite{HT-divisors}*{Theorem~07}, which states the following.
Let $x\ge z\ge y\ge 2$ be integers. Writing $\Theta(x,y,z)$ for the number of positive integers $n\le x$
such that the product of those prime factors of $n$ (allowing for multiplicities) that are less than $y$
exceeds $z$, we have
\[
\Theta(x,y,z)\ll x\exp\Big( -c_0\frac{\log z}{\log y}\Big),
\]
where $c_0$ and the implied constant are absolute.
We take
\begin{align*}
  x &= Q,\\
  \log y &= \log Q / (\log\log Q)^{2.5}, \\
  \log z &= \log Q / (\log\log Q)^2.
\end{align*}
The conclusion is that for all but $o(Q)$ integers $q\in[1,Q]$ we have
\[
  q=p_1p_2\cdots p_k q_0,
\]
where $p_1\ge\cdots\ge p_k$ are all the prime factors $p$ of $q$ such that
\[
  \log p > \log Q/(\log\log Q)^{2.5},
\]
and
\[
  \log q_0 \le \frac{\log Q}{(\log \log Q)^2}.
\]

The number of integers $q\in[1,Q]$ that are divisible by $p^2$ for some prime $p>\exp(\log Q/(\log\log Q)^3)$ can be estimated by
\[
\sum_{p=\exp(\log Q/(\log\log Q)^3)}^Q Q/p^2\le o(Q).
\]
Therefore we have $p_1 > \cdots > p_k$ almost surely.

Next, for any set of primes $T$, with harmonic sum
\[
  H(T) = \sum_{p \in T} \frac1p,
\]
if $\Omega(q, T)$ is the number of prime factors of $q$ in $T$, counted with multiplicity, then for all but $o(Q)$ integers $q \leq Q$ we have
\begin{equation} \label{OmegaT-estimate}
  \Omega(q, T) = H(T) + O(H(T)^{1/2} \log H(T))
\end{equation}
(and $\log$ can be replaced by any function tending to infinity). Indeed for any $k \geq 0$ we have a Poisson-type bound of the form
\[
  |\{q \leq Q : \Omega(q, T) = k\}| \ll Q e^{-H(T)} \frac{H(T)^k}{k!} \left(1 + \frac{k}{H(T)}\right)
\]
(see for example \cite{HT-divisors}*{Section~0.5}, \cite{ford-notes}*{Theorem 2.13}, or \cite{tudesq}*{Th\'eor\`eme 1}), and \eqref{OmegaT-estimate} follows by comparison with the Poisson distribution.

By \eqref{OmegaT-estimate} with
\[
  T_1 = \{p~\text{prime} : \log Q / (\log \log Q)^{2.5} < \log p \leq \log Q\},
\]
we find that
\[
  k \sim H(T_1) = 2.5 \log \log \log Q + O(1).
\]
Similarly, by \eqref{OmegaT-estimate} with
\[
  T_2 = \{p~\text{prime} : \log Q / (\log \log Q)^{3} < \log p \leq \log Q / (\log \log Q)^{2.5}\},
\]
we find that $q_0$ is almost surely divisible by some $p \in T_2$ (in fact some $c \log \log \log Q$ many $p$), so in particular
\[
  \log q_0 \geq \log Q / (\log \log Q)^3.
\]
This completes the proof.
\end{proof}

Let $E_1$ be the set of all $q$ which are exceptional in the sense of Lemma~\ref{lem:large-sieve}. Let $E_2$ be the set of all $q \leq Q$ which have some divisor $d \in E_1$ with $d > \log Q$. Then, by Lemma~\ref{lem:large-sieve} we have
\[
  |E_1 \cap [Q_1/2, Q_1]| \leq Q_1^{1-\eps/3} \qquad (\log Q \leq Q_1 \leq 2Q),
\]
so
\begin{align*}
  |E_2|
  &\leq \sum_{d \in E_1 \cap [\log Q, Q]} \frac{Q}{d} \\
  &\ll \sum_{\substack{Q_1 \in [\log Q, 2Q] \\ \text{dyadic} }} \frac{Q}{Q_1^{\eps/3}} \\
  &\ll_\eps \frac{Q}{(\log Q)^{\eps / 3}}.
\end{align*}
Hence $E_2$ is almost empty.

A typical integer $q$ can be written
\[
  q = p_1 p_2 \cdots p_k q_0
\]
as in Lemma \ref{lm:divisors}. Note that
\[
  d(q/q_0) = 2^k \ll \log \log q.
\]
We have a decomposition
\begin{align*}
  \mu_n \bmod q
  &= \sum_{s : q_0 \mid s \mid q} \Pop{s} \mu_n \\
  &= \piop{q_0} \mu_n + \sum_{\substack{s : q_0 \mid s \mid q \\ s \neq q_0}} \Pop{s} \mu_n.
\end{align*}
Hence
\[
  \|\mu_n \bmod q - u_q\|
  \leq \|\mu_n \bmod q_0 - u_{q_0}\| + \sum_{\substack{s : q_0 \mid s \mid q \\ s \neq q_0}} \|\Pop{s} \mu_n\|.
\]

If $q \notin E_2$ then each of the divisors $s$ with $q_0 \mid s \mid q$ is unexceptional in the sense of Lemma~\ref{lem:large-sieve}. There are two cases. If $q_0 < s^{\eps / 3}$ then Lemma~\ref{lem:large-sieve} applies, so
\[
  \|\Pop{s}\mu_n\| \ll \frac{\|\Pop{s}\|}{\log s} \leq \frac{d(s/q_0)}{\log q_0}.
\]
If on the other hand $q_0 \geq s^{\eps / 3}$, $q_0 \neq s$, then
\begin{align*}
  \Pop{s}\mu_n
  &= \sum_{q_1 : q_0 \mid q_1 \mid s} \mu(q_1 / q_0) \piop{q_1} \mu_n \\
  &= \sum_{q_1 : q_0 \mid q_1 \mid s} \mu(q_1 / q_0) (\piop{q_1} \mu_n - u_{q_1}),
\end{align*}
so, by Lemma~\ref{lem:small-moduli},
\[
  \|\Pop{s}\mu_n\| \leq d(s/q_0) \max_{q_1 \leq s} \|\mu_n \bmod q_1 - u_{q_1}\| \leq d(s/q_0) e^{-cn / \log s}.
\]
Note $\log s \ll \eps^{-1} \log q_0$. Similarly,
\[
  \|\mu_n \bmod q_0 - u_{q_0}\| \leq e^{-c n / \log q_0}.
\]

Putting the above estimates together, we have
\begin{align*}
  \|\mu_n\bmod q - u_q \|
  &\leq d(q / q_0) e^{-c\eps n / \log q_0} + \frac{d(q/q_0)^2}{ \log q_0 } \\
  &\leq e^{-c \eps (\log \log q)^2} + \frac{(\log \log q)^5}{\log q}.
\end{align*}
This completes the proof of Theorem~\ref{mainthm}(ii).

\section{Alternative argument based on sophisticated pruning}\label{sc:alternative}

In this section, we present an alternative argument to prove Theorem~\ref{mainthm}(ii) that is closer 
to that in \cite{CDG}*{Section~6}.
This also relies on the large sieve inequality, but it avoids the use of the projection operators 
introduced in the previous section.

For each $n$, let $\nu_n$ be a measure as in Corollary~\ref{cor:nu_n-exists} with $\alpha = 1/\log n$. Thus
\begin{align*}
  & \| \mu_n - \nu_n\| = n^{-1 + o(1)}, \\
  & \log \|\nu_n\|_2^{-2} = H n + o(n).
\end{align*}
Let $\pi_n$ be the following measure:
\begin{align*}
  \pi_n =& \mu_{w_0} 
  *[(m_{a^{w_0}})_* \nu_{w_1-C} *(m_{a^{w_0 + w_1-C}})_*\mu_C]\\
  &*[(m_{a^{w_0 + w_1}})_* \nu_{w_2-C} * (m_{a^{w_0 + w_1+w_2-C}})_*\mu_C]*\cdots\\
  &*[(m_{a^{w_0 + \cdots + w_{k-1}}})_* \nu_{w_k-C}*(m_{a^{w_0 +\cdots+ w_k-C}})_*\mu_C],
\end{align*}
where $w_0 = \floor{n / \log n}$, $C$ is a suitably large integer depending on $a$ and $\mu$ such that
\[
\supp\nu_{w-C}\subset[-a^w/2,a^w/2],
\]
and $w_1, \dots, w_k$ are chosen as in the following lemma.

\begin{lemma}
We can choose $w_1, \dots, w_k \geq n^{0.1}$ such that
\begin{enumerate}
  \item $k \sim (\log n)^2$,
  \item $w_1 + \cdots + w_k = n - w_0$, and
  \item for every $w \in [n^{0.2}, n]$ there is some $i$ such that \[w_1 + \cdots + w_i \in [(1-o(1))w, w].\]
\end{enumerate}
\end{lemma}
\begin{proof}
Let $k \sim (\log n)^2$, and let $a$ be such that
\[
  \exp((1+a) \log (n-w_0)/(k+a))\sim n^{0.15}.
\]
Choose $w_1, \dots, w_k$ so that
\[
  w_1 + \cdots + w_i = \floor{\exp((i+a) \log(n-w_0)/(k+a))}.
\]
Then
\[
  \frac{w_1 + \cdots + w_{i+1}}{w_1 + \cdots + w_i} \sim e^{1 / \log n} = 1 + O(1 / \log n),
\]
so every $w \in \{1, \dots, n\}$ can be approximated as claimed.
\end{proof}

It follows that
\[
  \|\mu_n - \pi_n\| \leq \sum_{i=1}^k \| \mu_{w_i-C} - \nu_{w_i-C}\| \ll n^{-0.1 + o(1)},
\]
and we also have the following property.
\begin{lemma}\label{lm:pi}
For every $w \in [n^{0.2}, n]$, we have a factorization
\[
  \pi_n = \mu_{w_0} * (m_{a^{w_0}})_* \pi * \pi'
\]
for some probability measures $\pi, \pi'$ such that $\supp \pi \subset \supp \mu_w$ and
\[
  \log \|\pi\|_2^{-2} = Hw + o(w).
\]
\end{lemma}
\begin{proof}
We take
\begin{align*}
\pi=& \nu_{w_1-C}*(m_{a^{w_1}})_* \nu_{w_2-C} * \cdots * (m_{a^{w_1 + \cdots + w_{i-1}}})_* \nu_{w_i-C}\\
\pi'=&(m_{a^{w_0+w_1 + \cdots + w_{i}}})_*\nu_{w_{i+1}-C} * \cdots * (m_{a^{w_0+w_1 + \cdots + w_{k-1}}})_* \nu_{w_k-C}\\
&*(m_{a^{w_0 + w_1-C}})_*\mu_C*\cdots*(m_{a^{w_0 +\cdots+ w_k-C}})_*\mu_C
\end{align*}
for a suitable choice of $i$ so that
\[
w_1+\cdots+w_i\in[(1-o(1))w,w].
\]
It follows by construction that $\supp \pi\subset\supp \mu_w$.

By the choice of $C$ in the definition of $\pi$,
it follows that the map
\begin{align*}
&\supp \nu_{w_1-C}\times\cdots
\times\supp  (m_{a^{w_1 + \cdots + w_{i-1}}})_* \nu_{w_i-C}&\to&\Z\\
&(x_1,\ldots,x_i) &\mapsto& x_1+\ldots+x_i
\end{align*}
is injective, so
\begin{align*}
  \|\pi\|_2^2
  &=\|\nu_{w_1-C}\|_2^2\cdots\|\nu_{w_i-C}\|_2^2 \\
  &=e^{-H(w_1-C) + o(w_1)} \cdots e^{-H(w_i-C) + o(w_i)} \\
  &=e^{-Hw + o(w) + iCH}
\end{align*}
and the claim follows since $w\geq n^{0.2}$ and $i \leq n^{0.1}$.
\end{proof}

\begin{proposition} \label{CDG-like-prop}
Assume $2H \geq (1+\eps)\log a$ and $Hn \geq (1+\eps) \log Q$. Then
\[
  \sum_{q \leq Q} q \|\pi_n \bmod q - u_q\|_2^2 \leq Q \exp( - c (\log Q)^{0.49}).
\]
\end{proposition}

The exponent $0.49$ can actually be improved to $0.5$. To do this, instead of $w_0 = \floor{n / \log n}$, take $w_0$ to be a sufficiently small (depending on $\eps$) constant multiple of $n$. We omit the details.

\begin{proof}
Write $Q = a^t$. We want to bound
\begin{align*}
  \sum_{q \in (a^{t-1}, a^t)} q \|\pi_n \bmod q - u_q\|_2^2
  &= \sum_{q \in (a^{t-1}, a^t)} \sum_{0 < k < q} |\pihat[n](k/q)|^2 \\
  &= \sum_{w=1}^t \sum_{r, s} M(r/s) |\pihat[n](r/s)|^2,
\end{align*}
where the last sum runs over all $r, s$ with
\begin{align*}
  &s \in (a^{w-1}, a^w), \\
  &(a, s) = 1, \\
  &r \in R(s) = \{r \bmod s : (r, s) = 1\},
\end{align*}
and $M(r/s)$ is the number of $k/q = r/s$ with $q \in (a^{t-1}, a^t)$ and $0 < k < q$. In other words, $M(r/s)$ is the number of multiples of $s$, coprime with $a$, in the range $(a^{t-1}, a^t)$. We have
\[
  M(r/s) \leq a^t / s,
\]
so our sum is bounded by
\[
  a^t \sum_{w=1}^t \sum_{r, s} \frac1s |\pihat[n](r/s)|^2.
\]
We split this sum into two parts depending on the size of $w$: either $w \leq n^{1/2} / \log n$ or $w > n^{1/2} / \log n$.

Consider first the lower range $w \leq n^{1/2} / \log n$. Since $\mu$ is not supported on a coset of a subgroup, $\muhat{}(\xi) = 1$ only if $\xi = 0$. Find $c > 0$ such that
\[
  |\muhat{}(\xi)|^2 \leq e^{-c} \qquad (\xi \notin (-1/a, 1/a)).
\]
For any $r \in R(s)$, $s \in (a^{w-1}, a^w)$, $r/s \notin (-a^{-w}, a^{-w})$, so for some $j \leq w-1$ we have $a^j r/s \notin (-1/a, 1/a)$, so
\[
  |\muhat[w](r/s)| = \prod_{j=0}^{w-1} |\muhat(a^j r/s)| \leq e^{-c}.
\]
Hence
\[
  |\muhat[n'w](r/s)|^2 \leq \left(\max_{r' \in R(s)} |\muhat[w](r'/s)|^2\right)^{n'} \leq e^{-cn'}
\]
for any $n'\in\Z_{>0}$.
Therefore
\[
  \sum_{s \in (a^{w-1}, a^w)} \frac1s \sum_{r \in R(s)} |\muhat[n'w](r/s)|^2 \leq a^{w} e^{-c n'}.
\]
Now because $\mu_{w_0}$ is a factor of $\pi_n$ we also have
\[
  \sum_{s \in (a^{w-1}, a^w)} \frac1s \sum_{r \in R(s)} |\pihat[n](r/s)|^2 \leq a^{w} e^{-c n'},
\]
provided that $n'w \leq w_0 \sim n / \log n$. Take $n' \sim w^{-1} n / \log n$. Then $a^w e^{-cn'}$ is negligible for $w < c'(n / \log n)^{1/2}$ for a sufficiently small constant $c'$.

Now consider the upper range $w > n^{1/2} / \log n$. Let $w' = \min(2w, n)$. By construction, $\pi_n$ has a factor $(m_{a^{w_0}})_* \pi$ such that
\[
  \supp \pi \subset \supp \mu_{w'} \subset [-O(a^{2w}), O(a^{2w})]
\]
and
\[
  \log\|\pi\|_2^{-2} = Hw' + o(w)
\]
by Lemma \ref{lm:pi}.
Therefore
\[
  |\pihat[n](r/s)|^2 \leq |\pihat(a^{w_0} r/s)|^2,
\]
and by the large sieve (Theorem \ref{th:large-sieve}),
\[
  \sum_{s \in (a^{w-1}, a^w)} \frac1s \sum_{r \in R(s)} |\pihat(r/s)|^2
  \ll a^w \|\pi\|_2^2 = a^w e^{-Hw' + o(w)}.
\]
By hypothesis
\begin{align*}
  2H &\geq (1+\eps) \log a, \\
  Hn &\geq (1+\eps) t \log a,
\end{align*}
so
\[
  H w' \geq (1+\eps) w \log a.
\]
Hence $a^w e^{-Hw' + o(w)}$ is bounded by $e^{-c w}$ for some $c>0$. Thus the contribution from this case is negligible too.
\end{proof}

We have proved that
\[
  \sum_{q \leq Q} q \|\pi_n \bmod q - u_q\|_2^2 \leq Q \exp(-c (\log Q)^{0.49}),
\]
provided that $Hn > (1+\eps) \log Q$. Hence apart from some
\[
  Q \exp(-c (\log Q)^{0.49})
\]
exceptions $q \leq Q$ we have
\[
  q \|\pi_n \bmod q - u_q\|_2^2 = \exp(-c (\log Q)^{0.49}),
\]
and hence
\begin{align*}
  \|\mu_n \bmod q - u_q\|
  &\leq \|\mu_n - \pi_n\| + q^{1/2} \|\pi_n \bmod q - u_q\|_2 \\
  &= (\log q)^{-c}.
\end{align*}
This proves Theorem~\ref{mainthm}(ii) (both prime and composite cases).

\section{The larger sieve argument}\label{sc:larger-sieve}

The purpose of this section is to prove Theorem \ref{larger-sieve-thm}.
The following argument is based on Gallagher's larger sieve as described in \cite{friedlander--iwaniec--opera}*{Chapter~9.7}. By Corollary~\ref{cor:nu_n-exists}, there is a measure $\nu_n$ such that
\begin{align*}
  & \|\mu_n - \nu_n\| = o(1), \\
  & \|\nu_n\|_2^{-2} = \exp((H+o(1))n).
\end{align*}
Let $Q = \|\nu_n\|_2^{-2}$. Note that
\[
  \|\nu_n \bmod p\|_2^2 = \frac1{Q} + \sum_{x \neq y} \nu_n(x) \nu_n(y) 1_{p \mid (x - y)}.
\]
Summing over primes $p \leq X$ with weight $\log p$, we have
\[
  \sum_{p \leq X} \|\nu_n \bmod p\|_2^2 \log p
  \leq \frac1{Q} \sum_{p \leq X} \log p + \sum_{x \neq y} \nu_n(x) \nu_n(y) \sum_{p \mid (x-y)} \log p.
\]
The first sum is $O(X)$. Since $\nu_n$ is supported on $[-O(a^n), O(a^n)]$, for $x, y \in \supp \nu_n$ we have
\[
  \sum_{p \mid (x - y)} \log p \leq \log |x-y| \ll n.
\]
Thus
\[
  \sum_{p \leq X} \|\nu_n \bmod p\|_2^2 \log p \ll \frac{X}{Q} + n.
\]
Let $B_1$ be the set of primes $p \leq Q$ such that
\[
  p \|\nu_n \bmod p\|_2^2 > \delta n^{1/2}.
\]
(Later we will take $\delta$ to be a sufficiently small constant.) Then it follows that
\[
  \sum_{p \in B_1} \frac{\log p}{p} \ll \delta^{-1} n^{1/2}.
\]
In the rest of the argument we assume $p \leq Q$ and $p \notin B_1$, i.e.,
\[
  p \|\nu_n \bmod p\|_2^2 \leq \delta n^{1/2}.
\]

Now we use a Fourier multiplicity argument, as in \cite{BV--ax+b}*{Section~3}. Define
\[
  \nu_n^{(m)} = \frac1{m+1} \sum_{i=0}^{m} \mu_i * (m_{a^i})_* \nu_n * (m_{a^{i+n}})_* \mu_{m-i}.
\]
Then
\[
  \|\mu_{n+m} - \nu_n^{(m)}\| \leq \|\mu_n - \nu_n\| = o(1),
\]
and
\[
  \hat{\nu_n^{(m)}}(\xi)
  = \frac1{m+1} \sum_{i=0}^m \muhat[i](\xi) \hat{\nu_n}(a^i \xi) \muhat[m-i](a^{i+n} \xi),
\]
so
\begin{align*}
  |\hat{\nu_n^{(m)}}(\xi)|
  &= \frac1{m+1} \sum_{i=0}^m |\muhat[i](\xi)| |\hat{\nu_n}(a^i \xi)| |\muhat[m-i](a^{i+n} \xi)| \\
  &\leq \frac1{m+1} \sum_{i=0}^m |\hat{\nu_n}(a^i \xi)| .
\end{align*}
By Cauchy--Schwarz,
\[
  |\hat{\nu_n^{(m)}}(\xi)|^2 \leq \frac1{m+1} \sum_{i=0}^m |\hat{\nu_n}(a^i \xi)|^2.
\]
Hence, provided $a$ has order at least $m+1$ in $\F_p^\times$,
\[
  \max_{r \neq 0} |\hat{\nu_n^{(m)}}(r/p)|^2 \leq \frac1{m+1} \sum_{r\neq 0} |\hat{\nu_n}(r/p)|^2 \leq \frac{\delta n^{1/2}}{m+1}.
\]
Applying this with $m = \delta n$, we get
\begin{equation} \label{notB2}
  \max_{r \neq 0} |\hat{\nu_n^{(m)}}(r/p)|^2 \leq n^{-1/2}.
\end{equation}

Let $B_2$ be the set of primes $p \leq Q$ such that $a$ does not have order at least $m+1$ in $\F_p^\times$. For $p \in B_2$, $p$ divides $a^i - 1$ for some $i \leq m$, so $|B_2| \leq m^2$. Thus
\[
  \sum_{p \in B_2} \frac{\log p}{p}
  \ll \sum_{p \leq m^2} \frac{\log p}{p}
  + |B_2| \frac{\log m}{m^2}
  \ll \log m.
\]
Thus we may assume $p \notin B_2$, so \eqref{notB2} holds.

Finally, observe that
\[
  \|\mu_{2n+m} - \nu_n * (m_{a^n})_*\nu_n^{(m)}\| \leq 2 \|\mu_n - \nu_n\| = o(1),
\]
and
\begin{align*}
  q\|\nu_n * (m_{a^n})_*\nu_n^{(m)} \bmod q - u_q\|_2^2
  &= \sum_{r\neq 0} |\hat{\nu_n}(r/p)|^2 |\hat{\nu_n^{(m)}}(a^nr/p)|^2 \\
  &\leq q\|\nu_n \bmod p\|_2^2 \max_{r\neq 0} |\hat{\nu_n^{(m)}}(r/p)|^2 \\
  &\leq \delta.
\end{align*}
Thus
\begin{align*}
  \|\mu_{2n+m} \bmod p - u_p\|
  &\leq \|\mu_{2n+m} - \nu_n * (m_{a^n})_*\nu_n^{(m)}\|\\
  &\qquad + p^{1/2}\|\nu_n * (m_{a^n})_*\nu_n^{(m)} \bmod p - u_p\|_2 \\
  &\leq \delta + o(1).
\end{align*}
In other words, for every $p \notin B_1 \cup B_2$ there is some
\[
  n' = 2n + m = (2 + \delta + o(1)) H^{-1} \log Q
\]
such that $\|\mu_{n'} \bmod q - u_q\| \leq \delta + o(1)$. Taking $\delta$ smaller than $\eps$, we deduce Theorem~\ref{larger-sieve-thm}.

\section{Open problems}\label{sc:OP}

It would be very interesting to understand the mixing time of the random walk if we drop the condition $H>\frac12\log a$ in
Theorem~\ref{mainthm}(ii).
We do not think this condition is necessary, but it is required in our proof so that we can apply the large sieve.
More specifically, we formulate the following open problem.

\begin{OP}
Let $\mu=\frac{1}{2}(\d_0+\d_1)$ and let $a\ge 2$ be an integer.
Is the mixing time of $X_n \bmod q$ equal to $(1+o(1))\log_2 q$ for almost all $q$ such that $(q, a)=1$?
\end{OP}

For $a=2$, the answer is yes trivially, because $X_n$ is uniformly distributed in the interval $[0,2^n-1]\cap \Z$.
The $a=3$ case is covered by our Theorem~\ref{mainthm}, and the answer is yes.
The $a=4$ case is not covered by our Theorem~\ref{mainthm}, but the answer is still yes, as can be seen by a modification
of our argument, which we sketch bellow.
The $a\ge 5$ case appears to be beyond the reach of the large sieve and we only have the bounds from Theorem~\ref{larger-sieve-thm}.

For the $a=4$ case of the above problem we modify the proof of Theorem~\ref{mainthm}(ii) as follows.
For simplicity, we only discuss the case of prime moduli.
Fix some $n\in\Z_{>0}$ and also some $\e=\e(n)>0$ to be specified below.
Write
\begin{align*}
\eta_1=&\mu_{\e n},\\
\eta_2=&(m_{4^{\e n}})_*\mu_{(1-\e)n}\\
\eta_3=&(m_{4^{n}})_*\mu_{\e n}.
\end{align*}
We choose $\e$ in such a way that $\e n$ is an integer,
$\e$ slowly goes to $0$ as $n$ grows.
Observe that
\begin{align*}
\|\eta_1*\eta_1*\eta_2\|_2^2<& 2^{-(1+c\e_1)n},\\
\supp(\eta_1*\eta_1*\eta_2)<&[-C 4^n,C4^n]
\end{align*}
with some absolute constants $c,C$.
Now we can apply the large sieve for the measure $\eta_1*\eta_1*\eta_2$
and conclude that
\[
\sum_{r\in R(p)}|\wh\eta_1(r/p)|^2\cdot|\wh\eta_1(r/p)|^2\cdot|\wh\eta_2(r/p)|^2\ll 2^n\cdot2^{-(1+c\e_1)n}=2^{-c\e_1 n}
\]
for most primes $p$ in the range $[2^{n-1},2^n]$.
Now the claimed mixing time bound follows from
\begin{align*}
\|\mu_{(1+\e)n} \bmod p -u_p\|_2^2
=&\sum_{r\in R(p)}|\wh\eta_1(r/p)|^2\cdot|\wh\eta_2(r/p)|^2\cdot|\wh\eta_3(r/p)|^2\\
\le&\Big(\sum_{r\in R(p)}|\wh\eta_1(r/p)|^4\cdot|\wh\eta_2(r/p)|^2\Big)^{1/2}\\
&\times\Big(\sum_{r\in R(p)}|\wh\eta_3(r/p)|^4\cdot|\wh\eta_2(r/p)|^2\Big)^{1/2}\\
\ll&2^{-c\e_1 n}.
\end{align*}
The second factor in the second to last line can be estimated similarly to the first one as we did above.

It would be also interesting  to obtain a better understanding of the exceptional moduli in Theorem~\ref{mainthm}.
We suggest the following two problems in this direction.

\begin{OP}
Fix some $\mu$ and $a\in\Z_{\ge 2}$.
Suppose that the base $a$ expansion of some $q$ is generic in a suitable sense.
Is the mixing time of $X_n \bmod q$ equal to $(1+o(1))H^{-1}\log q$?
\end{OP}

This question is somewhat vague.
It is motivated by the counter-example of $q=a^n-1$, for which the mixing time
is $O(\log q\log\log q)$ as was demonstrated in \cite{CDG}.

\begin{OP}
Fix some $\mu$, and write
$\{X_n^{(a)}\}_{n\in\Z_{>0}}$ for the random walk defined in the introduction with the parameter $a$.
Is it true that for each $q$ with $2,3\nmid q$ at least one of the following holds:
\begin{itemize}
\item the mixing time of $X_n^{(2)} \bmod q$ equals $(1+o(1))H(\mu,2)^{-1}\log q$,
\item the mixing time of $X_n^{(3)} \bmod q$ equals $(1+o(1))H(\mu,3)^{-1}\log q$?
\end{itemize}
\end{OP}

This question is motivated by a series of problems posed by Furstenberg on the joint action of
$x\mapsto 2 x$ and $x\mapsto 3 x$ on $\R/\Z$.

\section*{Acknowledgements}

We thank Kevin Ford for discussions related to Lemma~\ref{lm:divisors}.
We thank Martin Hildebrand and Max Wenqiang Xu for comments and suggestions on an earlier version of this paper.
We are grateful to the anonymous referee for carefully reading the paper and for helpful comments.

\bibliography{refs}

\begin{bibdiv}
\begin{biblist}

\bib{blmbook}{book}{
      author={Boucheron, St\'{e}phane},
      author={Lugosi, G\'{a}bor},
      author={Massart, Pascal},
       title={Concentration inequalities},
   publisher={Oxford University Press, Oxford},
        date={2013},
        ISBN={978-0-19-953525-5},
         url={https://doi.org/10.1093/acprof:oso/9780199535255.001.0001},
        note={A nonasymptotic theory of independence, With a foreword by Michel
  Ledoux},
      review={\MR{3185193}},
}

\bib{BV--ax+b}{article}{
      author={Breuillard, Emmanuel},
      author={Varj\'{u}, P\'{e}ter~P.},
       title={Cut-off phenomenon for the ax+b markov chain over a finite
  field},
        date={2019},
        note={Preprint, arXiv:1909.09053v1},
}

\bib{CDG}{article}{
      author={Chung, F. R.~K.},
      author={Diaconis, Persi},
      author={Graham, R.~L.},
       title={Random walks arising in random number generation},
        date={1987},
        ISSN={0091-1798},
     journal={Ann. Probab.},
      volume={15},
      number={3},
       pages={1148\ndash 1165},
  url={http://links.jstor.org/sici?sici=0091-1798(198707)15:3<1148:RWAIRN>2.0.CO;2-6&origin=MSN},
      review={\MR{893921}},
}

\bib{Der-sub-additive}{incollection}{
      author={Derriennic, Yves},
       title={Quelques applications du th\'{e}or\`eme ergodique sous-additif},
        date={1980},
   booktitle={Conference on {R}andom {W}alks ({K}leebach, 1979) ({F}rench)},
      series={Ast\'{e}risque},
      volume={74},
   publisher={Soc. Math. France, Paris},
       pages={183\ndash 201, 4},
      review={\MR{588163}},
}

\bib{friedlander--iwaniec--opera}{book}{
      author={Friedlander, John},
      author={Iwaniec, Henryk},
       title={Opera de cribro},
      series={American Mathematical Society Colloquium Publications},
   publisher={American Mathematical Society, Providence, RI},
        date={2010},
      volume={57},
        ISBN={978-0-8218-4970-5},
         url={https://doi.org/10.1090/coll/057},
      review={\MR{2647984}},
}

\bib{ford-notes}{misc}{
      author={Ford, Kevin},
       title={Anatomy of integers and random permutations course lecture
  notes},
  note={\url{https://faculty.math.illinois.edu/~ford/Anatomy_lectnotes.pdf}},
}

\bib{green--harper}{article}{
      author={Green, Ben},
      author={Harper, Adam~J.},
       title={Inverse questions for the large sieve},
        date={2014},
        ISSN={1016-443X},
     journal={Geom. Funct. Anal.},
      volume={24},
      number={4},
       pages={1167\ndash 1203},
         url={https://doi.org/10.1007/s00039-014-0288-1},
      review={\MR{3248483}},
}

\bib{hhms}{article}{
      author={Hare, K.~E.},
      author={Hare, K.~G.},
      author={Morris, B. P.~M.},
      author={Shen, W.},
       title={The entropy of {C}antor-like measures},
        date={2019},
        ISSN={0236-5294},
     journal={Acta Math. Hungar.},
      volume={159},
      number={2},
       pages={563\ndash 588},
         url={https://doi.org/10.1007/s10474-019-00962-1},
      review={\MR{4022144}},
}

\bib{hildebrand}{article}{
      author={Hildebrand, Martin},
       title={A lower bound for the {C}hung-{D}iaconis-{G}raham random
  process},
        date={2009},
        ISSN={0002-9939},
     journal={Proc. Amer. Math. Soc.},
      volume={137},
      number={4},
       pages={1479\ndash 1487},
         url={https://doi.org/10.1090/S0002-9939-08-09687-1},
      review={\MR{2465674}},
}

\bib{Hil2}{article}{
      author={Hildebrand, Martin},
       title={On a lower bound for the {C}hung-{D}iaconis-{G}raham random
  process},
        date={2019},
        ISSN={0167-7152},
     journal={Statist. Probab. Lett.},
      volume={152},
       pages={121\ndash 125},
         url={https://doi.org/10.1016/j.spl.2019.04.020},
      review={\MR{3953053}},
}

\bib{Hil-thesis}{book}{
      author={Hildebrand, Martin~Victor},
       title={Rates of convergence of some random processes on finite groups},
   publisher={ProQuest LLC, Ann Arbor, MI},
        date={1990},
  url={http://gateway.proquest.com/openurl?url_ver=Z39.88-2004&rft_val_fmt=info:ofi/fmt:kev:mtx:dissertation&res_dat=xri:pqdiss&rft_dat=xri:pqdiss:9035480},
        note={Thesis (Ph.D.)--Harvard University},
      review={\MR{2638852}},
}

\bib{Hil0}{article}{
      author={Hildebrand, Martin},
       title={Random processes of the form {$X_{n+1}=a_nX_n+b_n\pmod p$}},
        date={1993},
        ISSN={0091-1798},
     journal={Ann. Probab.},
      volume={21},
      number={2},
       pages={710\ndash 720},
  url={http://links.jstor.org/sici?sici=0091-1798(199304)21:2<710:RPOTF>2.0.CO;2-I&origin=MSN},
      review={\MR{1217562}},
}

\bib{HT-divisors}{book}{
      author={Hall, Richard~R.},
      author={Tenenbaum, G\'{e}rald},
       title={Divisors},
      series={Cambridge Tracts in Mathematics},
   publisher={Cambridge University Press, Cambridge},
        date={1988},
      volume={90},
        ISBN={0-521-34056-X},
         url={https://doi.org/10.1017/CBO9780511566004},
      review={\MR{964687}},
}

\bib{KV-RW}{article}{
      author={Ka\u{\i}manovich, V.~A.},
      author={Vershik, A.~M.},
       title={Random walks on discrete groups: boundary and entropy},
        date={1983},
        ISSN={0091-1798},
     journal={Ann. Probab.},
      volume={11},
      number={3},
       pages={457\ndash 490},
  url={http://links.jstor.org/sici?sici=0091-1798(198308)11:3<457:RWODGB>2.0.CO;2-5&origin=MSN},
      review={\MR{704539}},
}

\bib{Nev}{book}{
      author={Neville, Richard, III},
       title={On lower bounds of the {C}hung-{D}iaconis-{G}raham random
  process},
   publisher={ProQuest LLC, Ann Arbor, MI},
        date={2011},
        ISBN={978-1124-64149-2},
  url={http://gateway.proquest.com/openurl?url_ver=Z39.88-2004&rft_val_fmt=info:ofi/fmt:kev:mtx:dissertation&res_dat=xri:pqdiss&rft_dat=xri:pqdiss:3454522},
        note={Thesis (Ph.D.)--State University of New York at Albany},
      review={\MR{2873460}},
}

\bib{steele--efron--stein}{article}{
      author={Steele, J.~Michael},
       title={An {E}fron-{S}tein inequality for nonsymmetric statistics},
        date={1986},
        ISSN={0090-5364},
     journal={Ann. Statist.},
      volume={14},
      number={2},
       pages={753\ndash 758},
         url={https://doi.org/10.1214/aos/1176349952},
      review={\MR{840528}},
}

\bib{tudesq}{article}{
      author={Tudesq, Christian},
       title={Majoration de la loi locale de certaines fonctions additives},
        date={1996},
        ISSN={0003-889X},
     journal={Arch. Math. (Basel)},
      volume={67},
      number={6},
       pages={465\ndash 472},
         url={https://doi.org/10.1007/BF01270610},
      review={\MR{1418908}},
}

\end{biblist}
\end{bibdiv}

\end{document}